\pdfoutput=1
\RequirePackage{ifpdf}
\ifpdf % We~are running pdfTeX in pdf mode
\documentclass[pdftex]{sigma}%,draft
\else
\documentclass{sigma}
\fi

\usepackage[all]{xy}

\numberwithin{equation}{section}

\newtheorem{Theorem}{Theorem}[section]
\newtheorem*{Theorem*}{Theorem}
\newtheorem{Corollary}[Theorem]{Corollary}
\newtheorem{Lemma}[Theorem]{Lemma}
\newtheorem{Proposition}[Theorem]{Proposition}
 { \theoremstyle{definition}
\newtheorem{Definition}[Theorem]{Definition}

\newtheorem{Remark}[Theorem]{Remark} }
\newtheorem{Conjecture}[Theorem]{Conjecture}

\begin{document}
\allowdisplaybreaks

\newcommand{\arXivNumber}{2002.08620}

\renewcommand{\PaperNumber}{110}

\FirstPageHeading

\ShortArticleName{A Composite Order Generalization of Modular Moonshine}

\ArticleName{A Composite Order Generalization\\ of Modular Moonshine}

\Author{Satoru URANO}

\AuthorNameForHeading{S.~Urano}

\Address{Division of Mathematics, University of Tsukuba,\\ 1-1-1 Tennoudai, Tsukuba, Ibaraki 305-8571 Japan}
\Email{\href{mailto:usaranoto@gmail.com}{usaranoto@gmail.com}}

\ArticleDates{Received March 31, 2021, in final form December 21, 2021; Published online December 24, 2021}

\Abstract{We introduce a generalization of Brauer character to allow arbitrary finite length modules over discrete valuation rings. We show that the generalized super Brauer character of Tate cohomology is a linear combination of trace functions. Using this result, we find a~counterexample to a conjecture of Borcherds about vanishing of Tate cohomology for Fricke elements of the Monster.}


\Classification{11F22; 11F85; 17B69; 20C11; 20C20}

\section{Introduction}
Conway and Norton's monstrous moonshine conjecture \cite{CN} stated the existence of a graded representation $V=\bigoplus_{n \in \mathbb{Z}}V_n$ such that for any element $g$ of the monster group $\mathbb{M}$, the McKay--Thompson series $T_g(\tau)=\sum_{n \in \mathbb{Z}}\operatorname{Tr}(g|V_n)q^{n-1}$, $q={\rm e}^{2\pi {\rm i} \tau}$, is a Hauptmodul for some genus $0$ subgroup of ${\rm SL}_2(\mathbb{R})$ of the real numbers $\mathbb{R}$.
Frenkel, Lepowsky, and Meurman \cite{FLM} constructed a well-behaved graded representation as a vertex algebra, and Borcherds \cite{B92} showed that this vertex algebra satisfied Conway and Norton's conjecture.

The original modular moonshine conjectures of Ryba asserted the existence of vertex algebras~$^gV$ over finite fields with finite group actions, such that the graded Brauer characters are genus zero modular functions. He also suggested a construction of the vertex algebra $^{g}V$ (see Section~\ref{section3}) for an element $g \in \mathbb{M}$ of prime order~$p$ and of conjugacy class~$pA$ that is the largest conjugacy class of order~$p$, in terms of a self-dual integral form of the monster vertex algebra whose existence was proved later by Carnahan~\cite{C}.
Borcherds and Ryba reinterpreted the modular moonshine conjectures in terms of Tate cohomology with coefficients in an integral form~$V$ of the monster vertex algebra. They proved that the 1-st Tate cohomology $\hat{H}^1(g,V)$ for a~cyclic group $\langle g \rangle$ generated by any Fricke element $g$ (see Definition~\ref{def2.9}) of odd prime order vanishes. Borcherds proved that if $h$ is a $p$-regular element of $C_{\mathbb{M}}(g)$, then graded Brauer characters
$\widetilde{\operatorname{Tr}}\big(h|{\hat{H}}^i(g,V)\big)=\sum_{n \in \mathbb{Z}} \widetilde{\operatorname{Tr}}\big(h|{\hat{H}}^i(g,V_n)\big)q^{n-1}$, $i=0,1$, is a linear combination of Hauptmoduls for any non-Fricke element $g$ of odd prime order~\cite{B98,BR}. At the end of~\cite{B98}, Borcherds proposed a conjecture that asserts the existence of objects that unify modular moonshine with generalized moonshine. In his conjecture, conditions~2,~4 and~7 together imply $\hat{H}^1(g,V)$ vanishes for all Fricke elements~$g$. We will call this claim ``Borcherds's conjecture about vanishing of Tate cohomology''.

It is natural to consider Tate cohomology of~$V$ for elements of composite order, but there are no results in the literature. Tate cohomology does not yield vector spaces over a field when the element has composite order, but Brauer characters are only defined for such objects. Specifically, Tate cohomology yields torsion modules over the ground ring. We can't use the Brauer character for such modules because it is defined only for representations defined over a~field and $p$-regular elements. So, we need to define a generalized Brauer character.

In Section~\ref{section2}, we introduce some terminology in this paper. In Section~\ref{section3}, we summarize the modular moonshine conjecture for elements of prime order. In Section~\ref{section4}, we consider ramified extensions $R_p$ of the $p$-adic integers $\mathbb{Z}_p$ and calculate Tate cohomology. In Section~\ref{section5}, to generalize the modular moonshine conjectures, we define a generalized Brauer character and show that the generalized super Brauer character of Tate cohomology is a linear combination of traces. This is a generalization of \cite[Proposition~2.2]{BR} to the composite order case. Using this result, we give a counterexample to Borcherds's conjecture about vanishing of Tate cohomology. We propose a~weaker conjecture about vanishing of Tate cohomology.

\section{Definitions and basic notions}\label{section2}
 The symbols $\mathbb{Z}$, $\mathbb{Q}$, $\mathbb{R}$, $\mathbb{C}$, $\mathbb{F}_p$ will be used to denote the integers, rational numbers, real numbers, complex numbers, the finite field of prime~$p$ order, respectively.

 In this section, we introduce some terminology in this paper and will give brief definitions.

\begin{Definition}[\cite{C}]
Let $R$ be a commutative ring. A vertex algebra over $R$ is an $R$-module $V$ equipped with a distinguished element $\mathbf{1}\in V$, a distinguished linear transformation $T\colon V\rightarrow V$, and a left multiplication map
$Y(\cdot,z)\colon V\rightarrow(\operatorname{End} V)\big[\big[z,z^{-1}\big]\big]$, where $\operatorname{End} V$ is the endomorphisms on~$V$, written $Y(a,z)=\sum_{n\in\mathbb{Z}} a_nz^{-n-1}$ satisfying the following conditions:
 \begin{itemize}\itemsep=0pt
 \item For any $a \in V$, $Y(a,z)\mathbf{1}\in a+zV[[z]]$ and $Y\left(\mathbf{1},z\right)={\rm id}_V$, where ${\rm id}_V$ is the identity map on~$V$.
 \item For any $a,b \in V$, $Y(a,z)b\in V((z))$, i.e., $a_nb=0$ for $n$ sufficiently large.
 \item (The Jacobi identity) For any $a,b \in V$,
 \begin{gather*}
x^{-1}\delta \left( \frac{y-z}{x}\right) Y(a,y)Y(b,z)-x^{-1}\delta \left(\frac{z-y}{-x}\right)Y(b,z)Y(a,y)\\
\qquad{} =z^{-1}\delta \left(\frac{y-x}{z}\right)Y(Y(a,x)b,z),
 \end{gather*}
where $\delta (z)=\sum_{n\in \mathbb{Z}}z^n$, called the formal delta function.
 \item For any $a \in V$, $[T,Y(a,z)]=\frac{{\rm d}}{{\rm d}z}Y(a,z)$.
 \end{itemize}
$\mathbf{1}$ is called a vacuum vector and $Y(\cdot,z)$ is called a vertex operator.
\end{Definition}

\begin{Definition}
The Leech lattice $\Lambda$ is the unique free abelian group of rank 24 equipped with a symmetric bilinear form $\langle \cdot ,\cdot \rangle\colon \Lambda \times \Lambda \rightarrow \mathbb{Z}$ that satisfies the following conditions:
 \begin{itemize}\itemsep=0pt
 \item (even) For all $a \in \Lambda, \langle a,a \rangle \in 2\mathbb{Z}$.
 \item (positive definite) For all nonzero elements $a \in \Lambda$, $\langle a,a \rangle > 0$.
 \item (no roots) For all $a \in \Lambda$, $\langle a,a \rangle \neq 2$.
 \item (unimodular) If we choose a basis $(a_1,\dots,a_{24})$ of $\Lambda$, then we may form the Gram matrix $(\langle a_i,a_j \rangle)_{1\leq i,j\leq 24}$ and its determinant is~1.
 \end{itemize}
\end{Definition}

\begin{Definition}
A Lie algebra over a commutative ring $R$ is an $R$-module $\mathfrak{g}$ equipped with an $R$-bilinear map $[\cdot,\cdot]\colon \mathfrak{g}\times\mathfrak{g}\rightarrow \mathfrak{g}$ satisfying the following conditions:
 \begin{itemize}\itemsep=0pt
 \item For any $x \in \mathfrak{g}$, $[x,x]=0$.
 \item For any $x,y,z \in \mathfrak{g}$, $[x,[y,z]]+[y,[z,x]]+[z,[x,y]]=0$.
 \end{itemize}
\end{Definition}

 We review the construction of an integral form $V_\Lambda$ of a vertex algebra for the Leech lattice~$\Lambda$.
$\Lambda$ has a unique nonsplit central extension~$\hat{\Lambda}$ such that if we denote a lifting of $a \in \Lambda$ by ${\rm e}^a\in \hat{\Lambda}$, then ${\rm e}^a{\rm e}^b=\pm {\rm e}^{a+b}$ and ${\rm e}^a{\rm e}^b=(-1)^{\langle a,b\rangle}{\rm e}^b{\rm e}^a$ so we have an exact sequence
\[
0\rightarrow \{\pm1\}\rightarrow \hat{\Lambda}\rightarrow \Lambda \rightarrow 0.
\]
The group of automorphisms of $\hat{\Lambda}$ that preserve the inner product on $\Lambda$ is a nonsplit extension $2^{24}.\operatorname{Aut}(\Lambda)$ of the automorphism group $\operatorname{Aut}(\Lambda)$, and this group acts on $V_{\Lambda}$ faithfully.

A lattice vertex algebra is described by using notions of affine Lie algebras and twisted group rings. In this paper, we explain them briefly.

We consider a Lie algebra $\mathfrak{h}=\Lambda\otimes_\mathbb{Z}\mathbb{Q}$ with $[a,b]=0$ for any $a,b\in \mathfrak{h}$. Then we take the Lie algebra $\hat{\mathfrak{h}}=\mathfrak{h}\otimes \mathbb{Q}\big[t,t^{-1}\big]\oplus \mathbb{Q}c$ satisfying relations $\big[c,\hat{\mathfrak{h}}\big]=0$ and $[x\otimes t^m, y\otimes t^n]=\langle x,y \rangle m\delta_{m+n,0}c$, where $\delta_{i,j}$ is the Kronecker delta.
Let $\hat{\mathfrak{h}}^{\pm}$ be a subalgebra $\mathfrak{h}\otimes t^{\pm1}\mathbb{Q}\big[t^{\pm1}\big]$ and $\hat{\mathfrak{h}}^0$ be a subalgebra $\mathfrak{h}\otimes 1$. We view $\mathbb{Q}$ as $\hat{\mathfrak{h}}$-module which $\hat{\mathfrak{h}}^{\geq 0}$ acts as~0 and~$c$ acts as~1.

We view $\mathbb{Q}$ as $\{\pm1\}$-module which $\pm1$ acts as $\pm1$. The twisted group ring $\mathbb{Q}\{\Lambda \}$ of $\Lambda$ is defined to be $\mathbb{Q}\big[\hat{\Lambda}\big]\otimes_{\mathbb{Q}[\{ \pm 1\}]}\mathbb{Q}$ and is spanned by elements $\iota \big({\rm e}^\lambda\big):={\rm e}^\lambda \otimes 1$ for ${\rm e}^\lambda \in \hat{\Lambda}$. Then we have ${\rm e}^a\cdot\iota\big({\rm e}^\lambda\big)=\iota\big({\rm e}^a{\rm e}^\lambda\big)$ and $\{\pm1\}\cdot\iota\big({\rm e}^\lambda\big)=\pm\iota\big({\rm e}^\lambda\big)$, hence $\mathbb{Q}\{\Lambda \}$ is $\hat{\Lambda}$-module.

$V_\Lambda\otimes \mathbb{Q}$ is isomorphic to $\operatorname{Sym}_{\mathbb{Q}}\big(\hat{\mathfrak{h}}^-\big)\otimes \mathbb{Q}\{\Lambda \}$, where $\operatorname{Sym}_{\mathbb{Q}}(\cdot)$ means the symmetric algebra over $\mathbb{Q}$.
$\hat{\Lambda}$ and $\hat{\mathfrak{h}}$ act on $V_\Lambda\otimes \mathbb{Q}$ as following: For any $v \otimes \iota\big({\rm e}^\lambda\big) \in V_\Lambda$, $h\in \mathfrak{h}$, $n\neq 0$, and ${\rm e}^a \in \hat{\Lambda}$, writing $h\otimes t^n$ by $h(n)$,
\begin{gather*}
h(n)\cdot \big(v \otimes \iota\big({\rm e}^\lambda\big)\big)=h(n)v \otimes \iota\big({\rm e}^\lambda\big),\qquad h(0)\cdot \big(v \otimes \iota\big({\rm e}^\lambda\big)\big)=\langle h,\lambda \rangle \big(v \otimes \iota({\rm e}^\lambda)\big), \\
c\cdot\big(v \otimes \iota\big({\rm e}^\lambda\big)\big)=v \otimes \iota\big({\rm e}^\lambda\big),\\
 {\rm e}^a\cdot \big(v \otimes \iota\big({\rm e}^\lambda\big)\big)=v \otimes \iota\big({\rm e}^{a+\lambda}\big).
\end{gather*}
 $V_\Lambda\otimes \mathbb{Q}$ is graded by $\Lambda$, called $\Lambda$-degree, and we have $V_\Lambda\otimes \mathbb{Q}=\bigoplus_{\lambda \in \Lambda} \operatorname{Sym}_{\mathbb{Q}}\big(\hat{\mathfrak{h}}^-\big)\otimes \mathbb{Q}\iota \big({\rm e}^\lambda\big)$.

Let $\mathbf{1}$ be a vacuum vector of $V_{\Lambda}$ and let $Y(\cdot,z)$ be a vertex operator on $V_{\Lambda}$.
We define the operator $D^{(n)}$ by $Y(v,z)\mathbf{1}=\sum_{n\in \mathbb{Z}}D^{(n)}(v)z^n$ for any $v\in V_{\Lambda}$. The operator $D^{(n)}$ has the properties $D^{(0)}(v)=v$, $D^{(n)}D^{(m)}={n+m \choose n}D^{(n+m)}$, $D^{(n)}=0$ if $n<0$, and $D^{(n)}(vw)=\sum_{m\in \mathbb{Z}}D^{(m)}(v)D^{(n-m)}(w)$ for any $v,w \in V_\Lambda$. The derivation~$D$ is a element of $\hat{\mathfrak{h}}$ and has $\Lambda$-degree~0, i.e., $D^{(n)}(v)$ has the same $\Lambda$-degree as~$v$.

Suppose that we choose a basis $(a_1,\dots,a_{24})$ of $\Lambda$ and let~$p$ be a prime number. The sub\-ring~$V_{\Lambda,0}$ is generated as a commutative ring by ${\rm e}^{-a_j}D^{(i)}\big({\rm e}^{a_j}\big)$ for $i\geq 1$ and $1\leq j \leq 24$~\cite{BR}. Let $M$ be the free $\mathbb{Z}$-submodule of $\hat{\mathfrak{h}}_\mathbb{Q}$ given by the basis $\big\{{\rm e}^{-a_j}D^{(i)}\big({\rm e}^{a_j}\big)\big\}_{i\geq 1,1\leq j \leq 24}$. Then $V_{\Lambda}$ is isomorphic to $\operatorname{Sym}_{\mathbb{Z}}(M)\otimes \mathbb{Z}\{\Lambda\}$.

\begin{Definition}Let $R$ be a subring of the complex numbers $\mathbb{C}$. An $R$-vertex algebra $V$ is an $R$-form of the monster vertex algebra if $V\otimes_R\mathbb{C}\cong V^\natural$, where~$V^\natural$ is the monster vertex algebra over~$\mathbb{C}$.
\end{Definition}

\begin{Definition}
Let $G$ be a group and $N$ be an integer. An element $h$ of $G$ is $N$-regular if $(|h|,N)=1$.
\end{Definition}

\begin{Definition}
Let $R$ be a discrete valuation ring and $\mathfrak{m}_R$ the maximal ideal of $R$. Suppose that $k=R/\mathfrak{m}_R$ has prime characteristic $p$ and that the fraction field~$K$ of~$R$ has characteristic~$0$. The Brauer character of a $k[h]$-module~$M$ for a~$p$-regular element $h$ is defined by
\[
\widetilde{\operatorname{Tr}}(h|M)=\operatorname{Tr}\big(h|\widetilde{M}\otimes_RK\big),
\]
 where $\widetilde{M}$ is any $R$-free $R[h]$-module satisfying $M\cong\widetilde{M}\otimes_{R[h]}k[h]$.
\end{Definition}

\begin{Definition}[{\cite[Chapter IV]{CF}}]\label{Tate}
Suppose that $G$ is a group acting on an abelian group~$A$. Let ${\rm Nr}$ be the norm map $A\rightarrow A$; $a\mapsto\sum_{g\in G}{g(a)}$. Then, the $i$-th Tate cohomology ${\hat{H}}^i(G,A)$ is defined to be the following:
 \[
{\hat{H}}^i(G,A)= \begin{cases}
H_{-i-1}(G,A), & i\leq-2, \\
\operatorname{Ker}({\rm Nr})/\langle ga-a\,|\,a \in A, \, g\in G\rangle, & i=-1, \\
A^G/\operatorname{Im}({\rm Nr}), & i=0, \\
H^i(G,A), & i \geq1,
 \end{cases}
 \]
where $H_i(G,A)$ is the $i$-th homology, $H^i(G,A)$ is the $i$-th cohomology and $A^G$ is the largest submodule of $A$ on which $G$ acts trivially.
\end{Definition}

It has the following properties:
\begin{enumerate}\itemsep=0pt
\item[(1)] If $0\rightarrow A \rightarrow B \rightarrow C \rightarrow 0$ is a short exact sequence of $G$-modules, then we have a long exact sequence
\[
\cdots \rightarrow {\hat{H}}^{i-1}(G,C)\rightarrow {\hat{H}}^{i}(G,A)\rightarrow {\hat{H}}^{i}(G,B)\rightarrow {\hat{H}}^{i}(G,C)\rightarrow {\hat{H}}^{i+1}(G,A) \rightarrow \cdots.
\]
\item[(2)] $\bigoplus_{i\in \mathbb{Z}}{\hat{H}}^i(G,A)$ is $ |G |$-torsion module. In particular, if multiplication by $|G|$ is an isomorphism on $A$, then $\hat{H}^i(G,A)$ vanishes for all $i\in \mathbb{Z}$.
\end{enumerate}

Moreover, if $G$ is a cyclic group generated by $g$, then we have the following properties, where we denoted the Tate cohomology $\hat{H}^i(\langle g\rangle,A)$ for a cyclic group $\langle g\rangle$ by $\hat{H}^i(g,A)$.
\begin{enumerate}\itemsep=0pt
\item[(3)] ${\hat{H}}^i(g,A)\cong{\hat{H}}^{i+2}(g,A)$ for any $i\in\mathbb{Z}$.
\item[(4)]\label{Tsuper} ${\hat{H}}^\ast (g,A )=\hat{H}^0(g,A)\oplus\hat{H}^1(g,A)$ has more or less a super version of any $G$-invariant algebraic structure on~$A$. In particular, if~$V$ is a vertex algebra, then $\hat{H}^*(g,V)$ is a vertex superalgebra by the composite map
\[
\hat{H}^*(g,V)\otimes \hat{H}^*(g,V)\rightarrow \hat{H}^*(g,V\otimes V)\rightarrow \hat{H}^*(g,V((z)))\rightarrow \hat{H}^*(g,V)((z)).
\]
\item[(5)] If $A$ is a finite module, then $\big|{\hat{H}}^0(g,A)\big|=\big|{\hat{H}}^1(g,A)\big|$.
\end{enumerate}

\begin{Definition}
Let $\mathbb{H}$ be the upper half plane of $\mathbb{C}$. A Hauptmodul is a holomorphic function $f\colon \mathbb{H} \rightarrow \mathbb{C}$ satisfying the following conditions:
\begin{itemize}\itemsep=0pt
\item $f$ is invariant under some discrete subgroup $\Gamma_f$ of ${\rm SL}_2(\mathbb{R})$,
\item $f$ generates the meromorphic function field on $\mathbb{H}/\Gamma_f$.
\end{itemize}
\end{Definition}

\begin{Definition}\label{def2.9}
A monster element $g$ is Fricke if there is some $N\geq 1$ s.t.\ $T_g\big(\frac{-1}{N\tau}\big)=T_g(\tau)$ for any $\tau \in \mathbb{H}$.
\end{Definition}

\begin{Definition}
Let $R$ be a discrete valuation ring of mixed characteristic $(0,p)$, i.e., $R$ has chracteristic~0 and the residue field $R/\mathfrak{m}_R$ of $R$ has characteristic~$p$. Let~$S$ be the ring of integers of a~finite extension of the fraction field $\operatorname{Frac}(R)$ of $R$. Then $\mathfrak{m}^e_S=\mathfrak{m}_RS$ for some positive integer~$e$, called the ramification index. When $e=1$, $S$ is called an unramified extension of~$R$. On the other hand, when $e\neq 1$, $S$ is called a ramified extension of~$R$.
\end{Definition}

\begin{Definition}
Let $R$ be a commutative ring with prime characteristic~$p$. The Frobenius endomorphism $\operatorname{Frob}_p\colon R\rightarrow R$ is defined by $\operatorname{Frob}_p(r)=r^p$ for all $r\in R$.
\end{Definition}

If $R$ is a finite field, the Frobenius endomorphism is an automorphism.

\begin{Definition}
Let $A$ be a finite abelian group, $A^*$ be the homomorphisms $\operatorname{Hom}(A,\mathbb{C}^\times)$ from~$A$ to~$\mathbb{C}^\times$ and $f\colon A \rightarrow \mathbb{C}$ be a function. The discrete Fourier transform $\hat{f}\colon A^*\rightarrow \mathbb{C}$ is defined by $\hat{f}(\chi)=\sum_{a \in A}f(a)\chi (a)$ for any~$\chi \in A^*$.
\end{Definition}

There is an inverse discrete Fourier transform, it is described by the following
\[
f(a)=\frac{1}{|A|}\sum_{\chi \in A^*}\hat{f}(\chi)\overline{\chi}(a).
\]

\section{Modular moonshine for elements of prime order}\label{section3}

In this section, we summarize modular moonshine for elements of prime order \cite{B98,BR}.

Ryba proposed the modular moonshine conjecture. The existence of a self-dual integral form of the monster vertex algebra was not proved when he suggested it, but Carnahan proved it~\cite{C}.

\begin{Theorem}[the original modular moonshine conjecture]
Suppose that an element $g \in \mathbb{M}$ has prime order $p$ and lies in conjugacy class $pA$, and that $V$ is a self-dual integral form of the monster vertex algebra.
Let $V^g$ be the set of vectors fixed by $g$. We consider the following vertex algebra with characteristic~$p$:
 \[
{}^{g}V=\cfrac{V^g/pV^g}{(V^g/pV^g) \cap (V^g/pV^g)^{\perp}} .
\]
For $^{g}V=\bigoplus_{n \in \mathbb{Z}}{^{g}V_n}$, if $h$ is a $p$-regular element of the centralizer $C_{\mathbb{M}}(g)$ of $g$, then
\[
\sum_{n \in \mathbb{Z}} \widetilde{\operatorname{Tr}}(h|^{g}V_n)q^{n-1}=T_{gh}(\tau).
\]
 Here, we may take $R=\mathbb{Z}_p$ when computing Brauer character $\widetilde{\operatorname{Tr}}$.
 \end{Theorem}

Let~$G$ be a finite group including an element~$g$ of prime order. Borcherds and Ryba noticed that this definition of $^gV$ yields $\hat{H}^0(g,V)$. They redefined~$^gV$ to be $\hat{H}^*(g,V)$ and reinterpreted the conjectures in terms of Tate cohomology. To prove the conjectures, they used the fact that there are exactly~3 indecomposable finitely generated modules over the group ring $\mathbb{Z}_p[G]$ which are free as $\mathbb{Z}_p$-modules; there are $\mathbb{Z}_p$, $\mathbb{Z}_p[G]$ and $I$, which is the kernel of the natural map $\mathbb{Z}_p[G]\rightarrow \mathbb{Z}_p$.
By the following lemma, the Tate cohomology for any $\mathbb{Z}$-module is identified with the Tate cohomology for $\mathbb{Z}_p$-module.

\begin{Lemma}[{\cite[Lemma~2.1]{BR}}] \label{Zp}
If $A$ is a free $\mathbb{Z}[1/n]$-module acted on by a $p$-group $G$ with $(p,n)=1$, then the natural map from $\hat{H}^i(G,A)$ to $\hat{H}^i(G,A\otimes \mathbb{Z}_p)$ is an isomorphism for any $i\in \mathbb{Z}$.
\end{Lemma}

The next proposition is the main tool for calculating the Brauer character of a modular moonshine vertex super algebra.

\begin{Proposition}[{\cite[Proposition 2.2]{BR}}]\label{prop2.2}
Let $G$ be a finite group. Suppose that an element~$g$ of the center $\operatorname{Cent}(G)$ of $G$ has prime order $p$ and that $h \in G$ is a $p$-regular element. Let~$A$ be a finitely generated $\mathbb{Z}$-free $\mathbb{Z}[\langle g,h \rangle]$-module. Then $\hat{H}^*(g,A)=\hat{H}^0(g,A)-\hat{H}^1(g,A)$ is a virtual representation of $\langle h \rangle$, and
\[
\widetilde{\operatorname{Tr}}\big(h|\hat{H}^*(g,A)\big)=\operatorname{Tr}(gh|A) .
\]
\end{Proposition}

We have the following assertion by applying Proposition~\ref{prop2.2} to an element of the monster group and the monster vertex algebra.

\begin{Corollary}
Suppose that $g$ is an element of $\mathbb{M}$ of odd prime order. Let $V$ be an integral form of the monster vertex algebra. If $h\in C_{\mathbb{M}}(g)$ is $p$-regular, then $\sum_{n\in \mathbb{Z}}\widetilde{\operatorname{Tr}}\big(h|\hat{H}^*(g,V_n)\big)q^{n-1}$ is equal to the McKay--Thompson series $T_{gh}(\tau)$. In particular, it is a~Hauptmodul for some genus~$0$ subgroup of ${\rm SL}_2(\mathbb{R})$.
\end{Corollary}

Borcherds and Ryba proved that ${\hat{H}}^1(g,V)=0$ for any Fricke element $g$ of odd prime order. There are two main methods for proving ${\hat{H}}^1(g,V)=0$.

First, when $g$ has small order, we can transfer a vanishing statement from~$V_\Lambda$, using a decomposition of $V$ into submodules isomorphic to pieces of $V_\Lambda$~\cite{BR}.

\begin{Lemma}[{\cite[Lemma 4.1]{BR}}] \label{par}
If $A$ is a permutation module of a finite group $G$ over a~ring~$R$ with no~$|G|$ torsion, then $\hat{H}^{-1}(G,A)=0$.
$($A free $R$-module~$A$ acted on by a~group~$G$ is a~permutation module if it has a basis which is closed under the action of~$G$.$)$
\end{Lemma}

\begin{Lemma}[{\cite[Lemma 4.2]{BR}}]\label{closedpar}
The set of permutation modules is closed under taking sums, tensor products, and symmetric powers.
\end{Lemma}

\begin{Lemma}[{\cite[Lemma 4.5 and Theorem~4.6]{BR}}]\label{lattice}
If $g$ is an element of odd prime order in a~Mathieu group $M_{24}$ viewed as a subgroup of $\operatorname{Aut}(\Lambda)$, then $\hat{H}^1(g,\Lambda)=0$ and $\hat{H}^1(g,V_\Lambda)=0$.
\end{Lemma}

\begin{Theorem}[{\cite[Theorem 4.7]{BR}}]\label{pA}
Suppose that $g$ is an element of the monster of type $3A$, $3C$, $5A$, $7A$, $11A$, or $23A$. Then $\hat{H}^1(g,V)=0$.
\end{Theorem}

\begin{Remark}
Similarly, for an element $g$ of the monster type~$2A$, they proved $\hat{H}^1(g,V)=0$ by showing that $\hat{H}^1(g,\Lambda)$ and $\hat{H}^1(g,V_\Lambda)$ vanish. The method of proof is same as the case of odd prime order, but slightly more complicated.
\end{Remark}

Borcherds and Ryba don't use the assumption that $g$ has prime order in proofs of Lemma~\ref{lattice} and Theorem~\ref{pA}. However, this assumption is applied to use Lemma~\ref{Zp}. Hence, we will show that Lemma~\ref{Zp} holds without the assumption that~$G$ is a $p$-group, and that $\hat{H}^1(g,V)=0$ for an element $g$ of the monster of type~$15A$ or~$21A$ in Section~\ref{section4}.

 Second, when $g$ has large order, it suffices to show that the coefficients
\[
 c_g^-(n)=\dim\hat{H}^0(g,V_n)+\dim \hat{H}^1(g,V_n)
\]
 are equal to the coefficients
\[
 c_g^+(n)=\operatorname{Tr}(g|V_n)=\dim \hat{H}^0(g,V_n)-\dim \hat{H}^1(g,V_n).
\]
This was proved by combining Hodge theory, an improved ``no-ghost'' theorem, and some explicit computation.

Borcherds proved that graded Brauer characters
\[
\widetilde{\operatorname{Tr}}(h|{\hat{H}}^i(g,V))=\sum_{n \in \mathbb{Z}} \widetilde{\operatorname{Tr}}\big(h|{\hat{H}}^i(g,V_n)\big)q^{n-1}, \qquad i=0,1,
\]
 is a linear combination of Hauptmoduls for any non-Fricke element~$g$ of odd prime order~\cite{B98,BR}.

\begin{Theorem}[{\cite[Theorem 5.3]{BR}}]
If $g$ is an element of the monster of type~$2B$, then $\hat{H}^0(g,V_n)$ vanishes if $n$ is odd, and $\hat{H}^1(g,V_n)$ vanishes if~$n$ is even.
\end{Theorem}

\begin{Theorem}[{\cite[Theorem 6.1]{B98}}]
If $g$ is an element of the monster of type $3B$, $5B$, $7B$, and $13B$ and~$\sigma$ is the element of order~$2$ generating the center of $C_\mathbb{M}(g)/O_p(C_\mathbb{M}(g))$, where $O_p(C_\mathbb{M}(g))$ is the largest normal $p$-subgroup of $C_\mathbb{M}(g)$, then $O_p(C_\mathbb{M}(g))$ acts trivially on $^gV=\hat{H}^0(g,V)\oplus \hat{H}^1(g,V)$ and $\sigma$ fixes $\hat{H}^0(g,V)$ and acts as~$-1$ on $\hat{H}^1(g,V)$.
\end{Theorem}

We have the following complete description of $\widetilde{\operatorname{Tr}}\big(h|{\hat{H}}^i(g,V)\big)$, $i=0,1$,
\begin{gather*}
\widetilde{\operatorname{Tr}}\big(h|{\hat{H}}^0(g,V)\big)= \begin{cases}
T_{gh}(\tau), & g \in pA,3C,\vspace{1mm}\\
\dfrac{T_{gh}(\tau)+T_{gh}(\tau +1/2)}{2}, & g \in 2B, \vspace{1mm}\\
\dfrac{T_{gh}(\tau)+T_{gh\sigma}(\tau)}{2}, & g \in pB,2|(p-1),
\end{cases}
\\
 \widetilde{\operatorname{Tr}}\big(h|{\hat{H}}^1(g,V)\big)= \begin{cases}
0, & g \in pA,3C,\vspace{1mm}\\
\cfrac{-T_{gh}(\tau)+T_{gh}(\tau +1/2)}{2}, & g \in 2B, \vspace{1mm}\\
\cfrac{-T_{gh}(\tau)+T_{gh\sigma}(\tau)}{2}, & g \in pB,2|(p-1),
\end{cases}
\end{gather*}
where $\sigma \in C_{\mathbb{M}}(g)/O_p(C_{\mathbb{M}}(g))$ is an involution which acts as $1$ on $\hat{H}^0(g,V)$ and as $-1$ on $\hat{H}^1(g,V)$.

\section{Tate cohomology and ramified extensions}\label{section4}

In this section, we will prove properties of Tate cohomology related to ramified extensions and complete the proof of Theorem~4.7 of~\cite{BR} for an element of the monster of type~$15A $ or~$21A$.

Let $R$ be a discrete valuation ring. Let $g \in G$ be an element of order~$N$, $h \in G$ be an $N$-regular element and~$A$ be an $R[g]$-module.

\begin{Proposition}\label{tensor}
If a commutative ring homomorphism $R \rightarrow S$ is flat, then $\hat{H}^*(g,A\otimes_{R}S)\cong \hat{H}^*(g,A)\otimes_{R}S$.
\end{Proposition}

\begin{proof}Let ${\rm Nr}$ be a norm map as in Definition~\ref{Tate}. Then, $\hat{H}^0(g,A)\cong \operatorname{Ker}(g-1)/\operatorname{Im}({\rm Nr}).$ We have an exact sequence of $R$-modules $0\rightarrow \operatorname{Im}({\rm Nr})\rightarrow \operatorname{Ker}(g-1)\rightarrow \hat{H}^0(g,A)\rightarrow 0$. If $R \rightarrow S$ is flat, the sequence of $S$-modules $0\rightarrow \operatorname{Im}({\rm Nr})\otimes_R S\rightarrow \operatorname{Ker}(g-1)\otimes_R S\rightarrow \hat{H}^0(g,A)\otimes_R S\rightarrow 0$ is exact. We consider the following diagram:
\[ \xymatrix{
0\ar[r] & \operatorname{Im}({\rm Nr}\otimes_R {\rm id}_S)\ar[d] \ar[r] & \operatorname{Ker}((g-1)\otimes_R {\rm id}_S)\ar[d] \ar[r] & \hat{H}^0(g,A\otimes_R S)\ar[d]^{\psi_0} \ar[r] & 0 \\
 0\ar[r] & \operatorname{Im}({\rm Nr})\otimes_R S\ar[r] & \operatorname{Ker}(g-1)\otimes_R S\ar[r] & \hat{H}^0(g,A)\otimes_R S\ar[r] & 0.
 }
\]
Then, we have $\operatorname{Im}({\rm Nr}\otimes_R {\rm id}_S)\cong \operatorname{Im}({\rm Nr})\otimes_R S$ and $\operatorname{Ker}((g-1)\otimes_R {\rm id}_S)\cong \operatorname{Ker}(g-1)\otimes_R S$. Indeed, we can take the isomorphisms ${\rm Nr}\otimes_R {\rm id}_S(x\otimes s) \mapsto {\rm Nr}(x)\otimes s$ and $x\otimes s \mapsto x\otimes s$, respectively.
Also, the above diagram is commutative. Hence, by the Five lemma, $\psi_0$ is an isomorphism.

Like the degree 0 case, we have an exact sequence $0\rightarrow \operatorname{Im}(g-1)\rightarrow \operatorname{Ker}({\rm Nr})\rightarrow \hat{H}^1(g,A)\rightarrow 0$. As well above, we consider the following commutative diagram:
\[ \xymatrix{
0\ar[r] & \operatorname{Im}((g-1)\otimes_R {\rm id}_S)\ar[d] \ar[r] & \operatorname{Ker}({\rm Nr}\otimes_R {\rm id}_S)\ar[d] \ar[r] & \hat{H}^1(g,A\otimes_R S)\ar[d]^{\psi_1} \ar[r] & 0 \\
 0\ar[r] & \operatorname{Im}(g-1)\otimes_R S\ar[r] & \operatorname{Ker}({\rm Nr})\otimes_R S\ar[r] & \hat{H}^1(g,A)\otimes_R S\ar[r] & 0.
 }
\]
Then, we have $\operatorname{Im}((g-1)\otimes_R {\rm id}_S)\cong \operatorname{Im}(g-1)\otimes_R S$ and $\operatorname{Ker}({\rm Nr}\otimes_R {\rm id}_S)\cong \operatorname{Ker}({\rm Nr})\otimes_R S$. Indeed, we can take the isomorphisms $(g-1)\otimes_R {\rm id}_S(x\otimes s) \mapsto (gx-x)\otimes s$ and $x\otimes s \mapsto x\otimes s$, respectively.
By the Five lemma, $\psi_1$ is an isomorphism.
\end{proof}

Using the above proposition, we will prove that the 1-st Tate cohomology for an element of the monster type~$15A$ or~$21A$ vanishes. To do this, we will show the following lemma.

\begin{Lemma}\label{Zpkai}
If $A$ is a free $\mathbb{Z}[1/q]$-module acted on by a group $G$ with $|G|=p^kq $ and $(p,q)=1$, then the natural map from $\hat{H}^i(G,A)$ to $\hat{H}^i(G,A\otimes \mathbb{Z}_p)$ is an isomorphism for any $i\in \mathbb{Z}$.
\end{Lemma}

\begin{proof}
We show that multiplication by $|G|$ is an isomorphism on $\mathbb{Z}_p/\mathbb{Z}[1/q]$. For any $p^l\frac{a}{b} \in \operatorname{Ker}(|G|)$, where $l$ is a non-negative integer and $p$, $a$, $b$ are pairwise coprime, we have $|G|\cdot p^l\frac{a}{b}=p^{l+k}\frac{aq}{b} \in \mathbb{Z}[1/q]$. It follows that $b$ is a product of divisors of $q$ and then $p^l\frac{a}{b} \in \mathbb{Z}[1/q]$. Hence, $|G|$ is injective. For any $\sum_{i=0}^\infty a_ip^i \in \mathbb{Z}_p/\mathbb{Z}$, where $a_i \in \{0,\dots ,p-1\}$,we have $\sum_{i=0}^\infty a_ip^i \equiv \sum_{i=k}^\infty a_ip^i$ $({\rm mod}\ \mathbb{Z})$. Hence, we have $|G|\cdot \frac{1}{q}\sum_{i=0}^\infty a_{k+i}p^i=\sum_{i=k}^\infty a_ip^i$ and then $|G|$ is surjective on~$\mathbb{Z}_p/\mathbb{Z}$. By the surjectivity of $\mathbb{Z}_p/\mathbb{Z}\rightarrow \mathbb{Z}_p/\mathbb{Z}[1/q]$, $|G|$ is surjective on~$\mathbb{Z}_p/\mathbb{Z}[1/q]$. By property~(2) of Tate cohomology, $\hat{H}^i(G,A\otimes \mathbb{Z}_p/\mathbb{Z}[1/q])=0$ for all $i\in \mathbb{Z}$. For the exact sequence
\[
0\rightarrow A\rightarrow A\otimes \mathbb{Z}_p\rightarrow A\otimes (\mathbb{Z}_p/\mathbb{Z}[1/q])\rightarrow 0,
\]
 taking the Tate cohomology, we have a long exact sequence because of property~(1) of Tate cohomology. It follows that $\hat{H}^i(G,A)\cong \hat{H}^i(G,A\otimes \mathbb{Z}_p)$ for any $i\in \mathbb{Z}$.
\end{proof}

\begin{Lemma}\label{lattice15A}
If $g$ is an element of odd order in $M_{24}$ viewed as a subgroup of $\operatorname{Aut}(\Lambda)$, then $\hat{H}^1(g,\Lambda)=0$ and $\hat{H}^1(g,V_\Lambda)=0$.
\end{Lemma}

\begin{proof}
Let $|g|=p^kq$, where $p$ is a prime and $(p,q)=1$. By Proposition~\ref{tensor} and Lemma~\ref{Zpkai}, we have $\hat{H}^1(g,\Lambda)\otimes \mathbb{Z}[1/q]\cong \hat{H}^1(g,\Lambda \otimes \mathbb{Z}[1/q])\cong \hat{H}^1\big(g,\Lambda \otimes_{\mathbb{Z}[1/q]}\mathbb{Z}_p\big)$. By the same argument as the proof of Lemma~4.5 of~\cite{BR}, $\hat{H}^1\big(g,\Lambda \otimes_{\mathbb{Z}[1/q]}\mathbb{Z}_p\big)=0$. Hence, $\hat{H}^1(g,\Lambda)=0$.

Similarly, we have $\hat{H}^1(g,V_\Lambda)\otimes \mathbb{Z}[1/q]\cong \hat{H}^1\big(g,V_\Lambda \otimes_{\mathbb{Z}[1/q]}\mathbb{Z}_p\big).$ We can decompose $V_\Lambda \otimes \mathbb{Z}_p$ into the following:
\[
V_\Lambda \otimes \mathbb{Z}_p=\bigoplus_{\operatorname{orbits} \langle g \rangle \lambda}\bigg( \bigoplus_{g^r\lambda \in \langle g \rangle \lambda} V_{\Lambda ,g^r\lambda}\otimes \mathbb{Z}_p \bigg),
\]
where $\bigoplus_{{{\rm orbits} \langle g \rangle \lambda}}$ means the direct sum over distinct $\operatorname{orbits} \langle g \rangle \lambda$,
\[
\bigoplus_{g^r\lambda \in \langle g \rangle \lambda}V_{\Lambda,g^r\lambda}\otimes \mathbb{Z}_p\cong \mathbb{Z}_p\lbrack \langle g\rangle /\operatorname{Stab}_{\langle g \rangle}(\lambda)\rbrack \otimes_{\mathbb{Z}_p} (V_{\Lambda,0}\otimes \mathbb{Z}_p)
\]
 as $\mathbb{Z}_p[g]$-modules and $\operatorname{Stab}_{\langle g \rangle}(\lambda)$ is the stabilizer subgroup of $\langle g \rangle$ with respect to $\lambda$.

By the proof of Theorem~4.6 of~\cite{BR}, $V_{\Lambda,0}\otimes \mathbb{Z}_p$ is a permutation module. By Lemma~\ref{closedpar}, $\bigoplus_{g^r\lambda \in \langle g \rangle \lambda} V_{\Lambda ,g^r\lambda}\otimes \mathbb{Z}_p$ is a permutation module. By Lemmas~\ref{closedpar} and~\ref{par}, $V_\Lambda \otimes \mathbb{Z}_p$ is a permutation module and we have $\hat{H}^1(g,V_\Lambda)=0$.
\end{proof}

\begin{Theorem}\label{15A21A}
Suppose that $g$ is an element of the monster of type $15A$ or $21A$. Then $\hat{H}^1(g,V)=0$.
\end{Theorem}

\begin{proof}
In the same way as the proof of Theorem~4.7 of~\cite{BR}, this statement follows from Lemma~\ref{lattice15A}.
\end{proof}

Let $p$ be a prime factor of $N$ and $\zeta_N$ be a primitive $N$-th root of unity. Now, we consider a~ramified extension $R_p=\mathbb{Z}_p[\zeta_{N|h|}]$ of $\mathbb{Z}_p$.

We define rank 1 $R$-free $R_p[\langle g,h \rangle]$-modules $R_{n,m}$ with distinguished basis vectors $v_{n,m}\in R_{n,m}$, i.e.,
\[
R_{n,m}=\big\{cv_{n,m}\,|\,c\in R_p,\, gv_{n,m}=\zeta_N^nv_{n,m},\, hv_{n,m}=\zeta_{|h|}^mv_{n,m}\big\}.
\]
 Any $R_p[\langle g,h \rangle]$-module that is $R_p$-free of finite rank is an extension of rank~1 modules of the form~$R_{n,m}$ for $0\leq n\leq N-1$, $0\leq m\leq |h|-1$.
Hence, it suffices to consider the Tate cohomology of~$R_{n,m}$.

\begin{Lemma}\label{Tcoho}
For any $0\leq n \leq N-1$, $0\leq m \leq |h|-1$,
\begin{gather*}
\hat{H}^0(g,R_{n,m})= \begin{cases}
R_p/NR_p, & n=0, \\
0, & n\neq 0,
\end{cases}\\
\hat{H}^1(g,R_{n,m})= \begin{cases}
0, & n=0, \\
R_p/(1-\zeta_N^n)R_p, & n\neq 0,
\end{cases}
\end{gather*}
where $h$ acts as $\zeta_{|h|}^m$ on $R_p/NR_p$ and $R_p/\big(1-\zeta_N^n\big)R_p$.
\end{Lemma}

\begin{proof}
$R_{0,m}^g=R_{0,m}$, $R_{n,m}^g=0$, $n\neq 0$. Hence, $\hat{H}^0(g,R_{0,m})=R_p/NR_p$,
 $\hat{H}^0(g,R_{n,m})=0$, $n\neq 0$, because the norm map $R_{n,m} \rightarrow R_{n,m}$ is $v_{n,m} \mapsto \sum_{i=0}^{N-1}g^iv_{n,m}$.

By $\sum_{i=0}^{N-1}g^iv_{0,m}=Nv_{0,m}$, the kernel of the norm map on $R_{0,m}$ is~0. Hence, $\hat{H}^1(g,R_{0,m})=0$. By $\sum_{i=0}^{N-1}g^iv_{n,m}=\sum_{i=0}^{N-1}\zeta_N^{ni}v_{n,m}=0$, $n \neq 0$, the kernel of the norm map on~$R_{n,m}$ is~$R_{n,m}$. Therefore, by the definition of Tate cohomology, $\hat{H}^1(g,R_{n,m})=R_p/\big(1-\zeta_N^n\big)R_p$, $n\neq 0$.
\end{proof}

\section[p-Brauer character]{$\boldsymbol{p}$-Brauer character}\label{section5}

In this section, we will define a new notion of generalized Brauer character and prove properties of it. We will give a counterexample to a conjecture of Borcherds about vanishing of Tate cohomology.

Let $p$ be a prime number. Let $R$ be a discrete valuation ring of mixed characteristic $(0,p)$, $v\colon R\rightarrow \mathbb{Z}_{\geq 0}\cup \{ \infty \}$ be a~surjective valuation on $R$ and $K=\operatorname{Frac}(R)$.

We can't use the Brauer character for $R_p/(1-\zeta_N^n)R_p$ because this is not a vector space over a field. To solve this, we will define a $p$-Brauer character.

\begin{Definition}\label{def5.1}
Let $G$ be a finite group and $M$ be a finite length $R[G]$-module. Let $h$ be a $p$-regular element of $G$. Suppose that $M$ has a composition series $0=M_0\subsetneq M_1\subsetneq \dots \subsetneq M_n=M$ as an $R[h]$-module.
 We define the $p$-Brauer character of $M$ for $h$ by
\[
\widetilde{\operatorname{Tr}}_p(h|M)=\cfrac{1}{v(p)} \sum_{i=1}^{n}\operatorname{Tr} \big(h|\widetilde{M_i}\otimes K\big),
\]
 where $\widetilde{M_i}$ is an $R$-free $R[h]$-module satisfying $M_i/M_{i-1} \cong\widetilde{{M}_i}\otimes_{R[h]}R/\mathfrak{m}_R[h]$.
\end{Definition}

\begin{Lemma}
This definition does not depend on our choices of $M_i$ and $\widetilde{M}_i$.
\end{Lemma}

\begin{proof}
For $M_i/M_{i-1} \cong\widetilde{{M}_i}\otimes_{R[h]}R/\mathfrak{m}_R[h]\cong\widetilde{{M}_i'}\otimes_{R[h]}R/\mathfrak{m}_R[h]$, we have $\operatorname{Tr}\big(h|\widetilde{M}_i\otimes K\big)=\widetilde{\operatorname{Tr}}(h|M_i/M_{i-1})=\operatorname{Tr}\big(h|\widetilde{M}_i'\otimes K\big)$. By the Jordan--H\"older theorem, the composition series of~$M$ is unique up to permutation and isomorphism. That is, for any two composition series $0=M_0\subsetneq M_1\subsetneq \dots \subsetneq M_n=M$, $0=M_0'\subsetneq M_1'\subsetneq \dots \subsetneq M_{n'}'=M$, we have $n=n'$ and some permutation $\sigma$ of $\{1,\dots,n\}$ s.t.\ $M_i/M_{i-1}\cong M_{\sigma(i)}'/M_{\sigma(i)-1}'$. Hence, $\sum_{i=1}^{n}\widetilde{\operatorname{Tr}}(h|M_i/M_{i-1})=\sum_{i=1}^{n}\widetilde{\operatorname{Tr}}(h|M_{\sigma(i)}'/M_{\sigma(i)-1}') =\sum_{i=1}^n\widetilde{\operatorname{Tr}}(h|M_i'/M_{i-1}')$. Therefore the $p$-Brauer character is well-defined.
\end{proof}

Note that it does depend on a choice of a prime number~$p$.

We can consider the notion of an irreducible $p$-Brauer character, so that $M$ is a simple $R[h]$-module that $M\otimes_{R[h]} R/\mathfrak{m}_R[h]$ is a simple $R/\mathfrak{m}_R[h]$-module. In particular, an irreducible $p$-Brauer is an irreducible Brauer character and a character table of $p$-Brauer characters is the same as of Brauer characters.

We will prove that the $p$-Brauer character is additive, where ``additive'' means that for any short exact sequence $0\rightarrow M\rightarrow M'\rightarrow M''\rightarrow 0$, $\widetilde{\operatorname{Tr}}_p$ satisfies $\widetilde{\operatorname{Tr}}_p(h|M')=\widetilde{\operatorname{Tr}}_p(h|M)+\widetilde{\operatorname{Tr}}_p(h|M'')$.

\begin{Lemma}\label{additive}
The $p$-Brauer character is additive on short exact sequences.
\end{Lemma}

\begin{proof}
Let $0\rightarrow M\xrightarrow{\varphi} M'\xrightarrow{\psi} M''\rightarrow 0$ be an exact sequence of finite length $R[h]$-modules. Suppose that $M$ has length $m$ and a composition series $0=M_0\subsetneq M_1\subsetneq \dots \subsetneq M_m=M$. We use analogous notation for~$M'$ and~$M''$. Then the length~$m'$ of~$M'$ is $m+m''$ and
\[
0=\varphi(M_0)\subsetneq \varphi(M_1)\subsetneq \dots \subsetneq \varphi(M_m)=\psi^{-1}(M''_0)\subsetneq \psi^{-1}(M''_1)\subsetneq \dots \subsetneq \psi^{-1}(M''_{m''})=M'
\]
 is a composition series of $M'$. By injectivity of $\varphi$, $M_i\cong \varphi(M_i)$ for any $i \in \{0, \dots,m\}$ and by surjectivity of $\psi$, $\psi^{-1}(M''_j)/\operatorname{Ker}(\psi)\cong M''_j$ for any $j \in \{0, \dots,m''\}$. Hence, we have $\varphi(M_i)/\varphi(M_{i-1})\cong M_i/M_{i-1}$ and
\[
\psi^{-1}(M''_j)/\psi^{-1}(M''_{j-1})\cong \big(\psi^{-1}(M''_j)/\operatorname{Ker}(\psi)\big)/\big(\psi^{-1}(M''_{j-1})/\operatorname{Ker}(\psi)\big) \cong M''_j/M''_{j-1}.
\]
 Therefore,
\begin{align*}
\widetilde{\operatorname{Tr}}_p(h|M')&=\frac{1}{v(p)}\sum_{i=1}^m\widetilde{\operatorname{Tr}}\big(h|\varphi(M_i)/\varphi(M_{i-1})\big) +\frac{1}{v(p)}\sum_{j=1}^{m''}\widetilde{\operatorname{Tr}}\big(h|\psi^{-1}(M''_j)/\psi^{-1}(M''_{j-1})\big) \\
 &=\frac{1}{v(p)}\sum_{i=1}^m\widetilde{\operatorname{Tr}}(h|M_i/M_{i-1})+\frac{1}{v(p)}\sum_{j=1}^{m''}\widetilde{\operatorname{Tr}}(h|M''_j/M''_{j-1}) \\
 &=\widetilde{\operatorname{Tr}}_p(h|M)+\widetilde{\operatorname{Tr}}_p(h|M'').\tag*{\qed}
\end{align*}\renewcommand{\qed}{}
\end{proof}

\begin{Lemma}\label{tensorBch}
Let $S$ be the ring of integers of a finite extension of $\operatorname{Frac}(R)$ and $A$ be a finite length $R[h]$-module. Then $\widetilde{\operatorname{Tr}}_p(h|A\otimes_{R}S)=\widetilde{\operatorname{Tr}}_p(h|A)$.
\end{Lemma}

\begin{proof}
We have that the $p$-Brauer character of $A$ is described the sum of $p$-Brauer characters of composition factors of $A$ by Definition~\ref{def5.1}. Hence, it suffices to consider the case $A$ is a simple $R[h]$-module on which $h$ acts as $\zeta_{|h|}^m$. Now,
\[
\widetilde{\operatorname{Tr}}_p(h|A)=\frac{1}{v_R(p)}\operatorname{Tr} \big(h|\widetilde{A}\otimes \operatorname{Frac}(R)\big),
\]
where $v_R$ is a discrete valuation on~$R$, $\widetilde{A}$ is an $R$-free $R[h]$-module satisfying $A\cong \widetilde{A}\otimes_{R[h]} R/\mathfrak{m}_R[h]$ and $\mathfrak{m}_R$ is a maximal ideal of~$R$. We have
\[
A\otimes_R S\cong \widetilde{A}\otimes_{R[h]} R/\mathfrak{m}_R[h]\otimes_R S\cong \widetilde{A}\otimes_{R[h]} S/\mathfrak{m}_RS[h].
\]

 Let $v_S$ be a discrete valuation on $S$ and $\mathfrak{m}_S$ be a maximal ideal of~$S$. If $S$ is an unramified extension, then $\mathfrak{m}_S=\mathfrak{m}_RS$ and $v_S(p)=v_R(p)$. Hence,
 \begin{align*}
\widetilde{\operatorname{Tr}}_p(h|A\otimes_R S) &=\frac{1}{v_S(p)}\operatorname{Tr}\big(h|\widetilde{A}\otimes_R S\otimes \operatorname{Frac}(S)\big)\\
& =\frac{1}{v_R(p)}\operatorname{Tr}\big(h|\widetilde{A}\otimes \operatorname{Frac}(R)\big)=\widetilde{\operatorname{Tr}}_p(h|A).
 \end{align*}

 If $S$ is a ramified extension, then there is a positive integer $k>1$ s.t.\ $\mathfrak{m}_S^k=\mathfrak{m}_RS$ and $v_S(p)=kv_R(p)$. Hence, we have $A\otimes_R S=\widetilde{A}\otimes_{R[h]}S/\mathfrak{m}_S^k[h]$ and a composition series $0\subsetneq \widetilde{A}\otimes S/\mathfrak{m}_S\subsetneq \dots \subsetneq \widetilde{A}\otimes S/\mathfrak{m}_S^k$. Its composition factors are all isomorphic to $\widetilde{A}\otimes_{R[h]} S/\mathfrak{m}_S[h]$. Therefore,
\begin{align*}
\widetilde{\operatorname{Tr}}_p(h|A\otimes_R S) & =\frac{1}{v_S(p)}\sum_{i=1}^k\operatorname{Tr} \big(h|\widetilde{A}\otimes_R S\otimes \operatorname{Frac}(S)\big)\\
& =\frac{1}{kv_R(p)}k\operatorname{Tr}\big(h|\widetilde{A}\otimes \operatorname{Frac}(R)\big) =\widetilde{\operatorname{Tr}}_p(h|A).\tag*{\qed}
\end{align*}\renewcommand{\qed}{}
\end{proof}

By Proposition \ref{tensor} and Lemma~\ref{tensorBch}, any finitely generated $\mathbb{Z}$-free $\mathbb{Z}[\langle g,h \rangle]$-module base changes to an $R_p[\langle g,h \rangle]$-module that is $R_p$-free of finite rank and the $p$-Brauer character of the Tate cohomology is unchanged.
By Lemma \ref{additive}, it suffices to consider the $p$-Brauer character for the Tate cohomology of~$R_{n,m}$.

We will calculate the Tate cohomology of $R_{n,m}$ by using the $p$-Brauer character. To do this, we use the following lemma:

\begin{Lemma}[{\cite[Chapter III, Lemma 3]{CF}}] \label{pR}
 Let $q=p^k$ for a positive integer~$k$. Then $pR=(1-\zeta_q)^{\varphi(q)}R$, where $\varphi$ is Euler's totient function.
\end{Lemma}

For a positive integer $s$ and $R=\mathbb{Z}_{p}[\zeta_{p^s}]$, a discrete valuation $v(p)$ is given by $\varphi(p^s)$ and maximal ideal of~$R$ is $(1-\zeta_{p^s})R$.

\begin{Lemma}\label{Bch}
Suppose that the prime factorization of $N$ is $\prod p_i^{n_i}$.
Then,
\[
\widetilde{\operatorname{Tr}}_{p_i}(h|\hat{H}^*(g,R_{n,m}))= \begin{cases}
n_i\zeta_{|h|}^m, & n=0, \vspace{1mm}\\
\dfrac{-\zeta_{|h|}^m}{\varphi \big(p_i^l\big)}, & |g^n|=p_i^l, \ 1\leq l\leq n_i, \vspace{1mm}\\
0, & otherwise.
\end{cases}.
\]
\end{Lemma}

\begin{proof}Now, $h$ acts as $\zeta_{|h|}^m$ on $\hat{H}^*(g,R_{n,m})$ and $v(p)=\varphi\big(p_i^{n_i}\big)$. By \[
\widetilde{\operatorname{Tr}}_{p_i}\big(h|\hat{H}^*(g,R_{n,m})\big)=\widetilde{\operatorname{Tr}}_{p_i}\big(h|\hat{H}^0(g,R_{n,m})\big) -\widetilde{\operatorname{Tr}}_{p_i}\big(h|\hat{H}^1(g,R_{n,m})\big)
\]
 and Lemma~\ref{Tcoho}, it suffices to calculate the $p_i$-Brauer characters of $R_{p_i}/NR_{p_i}$ and $R_{p_i}/(1-\zeta_N^n)R_{p_i}$.

 When $n=0$, by Lemmas~\ref{Tcoho} and~\ref{pR},
\[
 R_{p_i}/NR_{p_i}\cong R_{p_i}/p_i^{n_i}R_{p_i}\cong R_{p_i}/\big(1-\zeta_{p_i^{n_i}}\big)^{n_i\varphi(p_i^{n_i})}R_{p_i}.
\]
This module has a composition series
\[
0\subsetneq R_{p_i}/(1-\zeta_{{p_i}^{n_i}})R_{p_i}\subsetneq \dots \subsetneq R_{p_i}/\big(1-\zeta_{{p_i}^{n_i}}\big)^{n_i\varphi(p_i^{n_i})}R_{p_i}
\]
 and its composition factors are all isomorphic to $R_{p_i}/\big(1-\zeta_{{p_i}^{n_i}}\big)R_{p_i}$. Hence,
\[
 \widetilde{\operatorname{Tr}}_{p_i}(h|R_{p_i}/NR_{p_i})=n_i\zeta_{|h|}^m.
\]

 When $|g^n|=p_i^l$, $\zeta_N^n$ is a $p_i^l$-th primitive root of unity. By Lemma~\ref{pR}, we have
\[
 \big(1-\zeta_{p_i^l}\big)^{\varphi(p_i^l)}=\big(1-\zeta_{p_i^{n_i}}\big)^{\varphi(p_i^{n_i})}.
\]
 Hence, by Lemma~\ref{Tcoho}, we have
\[
R_{p_i}/\big(1-\zeta_{{p_i}^l}\big)R_{p_i}\cong R_{p_i}/\big(1-\zeta_{{p_i}^{n_i}}\big)^{\varphi(p_i^{n_i})/\varphi(p_i^l)}R_{p_i}.
\]
 This module has the length $\varphi\big(p_i^{n_i}\big)/\varphi\big(p_i^l\big)$ and its composition factors are all isomorphic to $R_{p_i}/\big(1-\zeta_{{p_i}^{n_i}}\big)R_{p_i}$. Calculating the $p_i$-Brauer character, we have
\[
 \widetilde{\operatorname{Tr}}_{p_i}\big(h|R_{p_i}/\big(1-\zeta_{p_i^l}\big)R_{p_i}\big)=1/\varphi\big(p_i^l\big).
\]

 When $|g^n|=p_i^kq$, where $q\neq 1$, $(q,p_i)=1$ and $0\leq k\leq n_i$, $\zeta_N^n$ is a $p_i^kq$-th primitive root of unity. Let $\pi$ be a surjective homomorphism $R_{p_i}\rightarrow \mathbb{F}_{p_i}\big[\zeta_{N|h|/p_i^{n_i}}\big]$; $\sum_{j=0}^\infty c_jp_i^j \mapsto c_0$, $\zeta_{N|h|} \mapsto \zeta_{N|h|/p_i^{n_i}}$. Then $\operatorname{Ker}(\pi)$ is the maximal ideal of $R_{p_i}$ and $\operatorname{Frob}_{p_i}$ on $\mathbb{F}_{p_i}[\zeta_{N|h|/p_i^{n_i}}]$ is an automorphism. We have $\pi\big(\zeta_{p_i^kq}\big)=\operatorname{Frob}_{p_i}^{-k}\big(\zeta_{p_i^kq}^{p_i^k}\big)\neq 1$.
 Hence, $1-\zeta_N^n \notin \operatorname{Ker}(\pi)$. $1-\zeta_N^n$ is a unit of $R_{p_i}$ because every element of a local ring not included in a maximal ideal is a unit. Therefore, $R_p/\big(1-\zeta_N^n\big)R_p=0$.
\end{proof}

We generalize Proposition~\ref{prop2.2} \cite[Proposition~2.2]{BR} using the $p$-Brauer character.

\begin{Theorem} \label{2.2kai}
Let $G$ be a finite group. Suppose that $g \in \operatorname{Cent}(G)$ has order $N$ and that $h \in G$ is an $N$-regular element. Let~$A$ be a finitely generated $R$-free $R[\langle g,h \rangle]$-module. Then $\hat{H}^*(g,A)=\hat{H}^0(g,A)-\hat{H}^1(g,A)$ is a virtual representation of $\langle h \rangle$, and
\begin{gather*}
\widetilde{\operatorname{Tr}}_{p_i}(h|\hat{H}^*(g,A))=\sum_{k=1}^{N-1}a_{k,p_i}\operatorname{Tr}\big(g^kh|A\big) ,
\\
a_{k,p_i}=\begin{cases}
\dfrac{1}{\sum_{(p_i,d)=1,d|N}\varphi (N/d)}\left(n_i-l-\dfrac{n_i-l-1}{p_i}\right), & \big(k,p_i^{n_i}\big)=p_i^l, \ 0\leq l \leq n_i-1, \\
0, & \big(k,p_i^{n_i}\big)=p_i^{n_i},
\end{cases}
\end{gather*}
where the prime factorization of $N$ is $\prod p_i^{n_i}$.
\end{Theorem}

\begin{proof}
Suppose
\[
\widetilde{\operatorname{Tr}}_{p_i}\big(h|\hat{H}^*(g,A)\big)=\sum_{k=1}^{N-1}a_{k,p_i}\operatorname{Tr}\big(g^kh|A\big). \]
It suffices to take $A=R_{n,m}$.
By an inverse discrete Fourier transform and Lemma~\ref{Bch}, we have
\begin{align*}
a_{k,p_i}&=\frac{1}{N}\sum_{b=0}^{N-1}\widetilde{\operatorname{Tr}}_{p_i}\big(h|\hat{H}^*(g,R_{b,m})\big)\overline{\operatorname{Tr}\big(g^kh|R_{b,m}\big)} \\
 &=\frac{1}{N}\Bigg(n_i+\sum_{|g^b|=p_i} \frac{-1}{\varphi (p_i)}\zeta_N^{-kb}+\sum_{ |g^b|=p_i^2} \frac{-1}{\varphi \big(p_i^2\big)}\zeta_N^{-kb}+\dots+\sum_{|g^b|=p_i^{n_i}} \frac{-1}{\varphi \big(p_i^{n_i}\big)}\zeta_N^{-kb}\Bigg).
\end{align*}

Let $\big(k,p_i^{n_i}\big)=p_i^l$. If $s>l$, then $\sum_{|g^b|=p_i^s}\zeta_N^{-kb}$ is a sum of primitive $p_i^{s-l}$-th roots of unity and equals an integer multiple of the M\"obius function $\mu\big(p_i^{s-l}\big)$. If $s\leq l$, then we have
\[
\sum_{|g^b|=p_i^s}\zeta_N^{-kb}=\varphi(p_i^s).
\]
Hence, when $\big(k,p_i^{n_i}\big)=p_i^{n_i}$, we have 
\[ a_{k,p_i}=\frac{1}{N}(n_i-1-\dots-1)=0.\]
When $\big(k,p_i^{n_i}\big)=p_i^{l}$, because $\big(N/p_i^{n_i}\big)=\sum_{d|(N/p_i^{n_i})}\varphi\big(N/p_i^{n_i}d\big)$, we have
\begin{align*}
a_{k,p_i}&=\frac{1}{N}\Bigg(n_i-1-\dots-1+\frac{p_i^l}{\varphi \big(p_i^{l+1}\big)}\Bigg)\\
 &=\frac{(n_i-l)\varphi \big(p_i^{l+1}\big)+p_i^l}{N\varphi \big(p_i^{l+1}\big)} =\frac{p_i^{l+1}(n_i-l)-p_i^l(n_i-l-1)}{p_i^{l+1}\varphi\big(p_i^{n_i}\big)\big(N/p_i^{n_i}\big)}\\
 &=\frac{1}{\sum_{(p_i,d)=1,d|N}\varphi (N/d)}\bigg(n_i-l-\frac{n_i-l-1}{p_i}\bigg).\tag*{\qed}
\end{align*}\renewcommand{\qed}{}
\end{proof}

\begin{Corollary} \label{newHaup}
Suppose that $g$ is an element of $\mathbb{M}$ of order $N=\prod p_i^{n_i}$. Let $V$ be an integral form of the monster vertex algebra. If $h\in C_{\mathbb{M}}(g)$ is $N$-regular, then $\sum_{n\in \mathbb{Z}}\widetilde{\operatorname{Tr}}_{p_i}\big(h|\hat{H}^*(g,V_n\otimes \mathbb{Z}_{p_i})\big)q^{n-1}$ is a linear combination of Hauptmoduls and is denoted by $\sum_{d|N, d\neq N}a_{d,p_i}\varphi (N/d)T_{g^dh}(\tau).$
\end{Corollary}

\begin{proof}
By Theorem~\ref{2.2kai} and the fact that McKay--Thompson series depend only on conjugacy classes.
\end{proof}

For example, we consider a case that $g\in \mathbb{M}$ has order $N=15=3\cdot 5$. Then by using the 3-Brauer character, we have \begin{gather*}
a_{1,3}=\frac{1}{\varphi(15)+\varphi(15/5)}\left(1-0-\frac{1-0-1}{3}\right)=\frac{1}{10}, \qquad a_{3,3}=0, \\ a_{5,3}=\frac{1}{\varphi(15)+\varphi(15/5)}\left(1-0-\frac{1-0-1}{3}\right)=\frac{1}{10}.
\end{gather*} Hence, for any 15-regular element $h\in C_{\mathbb{M}}(g)$, we have
\[
\sum_{n\in \mathbb{Z}}\widetilde{\operatorname{Tr}}_3\big(h|\hat{H}^*(g,V_n\otimes \mathbb{Z}_3)\big)q^{n-1}= \frac{4}{5}T_{gh}(\tau)+\frac{1}{5}T_{g^5h}(\tau).
\]
Similarly, by using the 5-Brauer character, we have
\begin{gather*}
a_{1,5}=\frac{1}{\varphi(15)+\varphi(15/3)}\left(1-0-\frac{1-0-1}{5}\right)=\frac{1}{12},\\
 a_{3,5}=\frac{1}{\varphi(15)+\varphi(15/3)}\left(1-0-\frac{1-0-1}{5}\right)=\frac{1}{12}, \qquad a_{5,5}=0.
\end{gather*}
 Hence, for any 15-regular element $h\in C_{\mathbb{M}}(g)$, we have
\[
\sum_{n\in \mathbb{Z}}\widetilde{\operatorname{Tr}}_5\big(h|\hat{H}^*(g,V_n\otimes \mathbb{Z}_5)\big)q^{n-1}
= \frac{2}{3}T_{gh}(\tau)+\frac{1}{3}T_{g^3h}(\tau).
\]
In particular, for $g\in 15A$, we have the following relation from Theorem \ref{15A21A}:
\begin{gather*}
\sum_{n\in \mathbb{Z}}\widetilde{\operatorname{Tr}}_3\big(h|\hat{H}^0(g,V_n\otimes \mathbb{Z}_3)\big)q^{n-1}= \frac{4}{5}T_{gh}(\tau)+\frac{1}{5}T_{g^5h}(\tau),\\
\sum_{n\in \mathbb{Z}}\widetilde{\operatorname{Tr}}_3\big(h|\hat{H}^1(g,V_n\otimes \mathbb{Z}_3)\big)q^{n-1}=0,\\
\sum_{n\in \mathbb{Z}}\widetilde{\operatorname{Tr}}_5\big(h|\hat{H}^0(g,V_n\otimes \mathbb{Z}_5)\big)q^{n-1}= \frac{2}{3}T_{gh}(\tau)+\frac{1}{3}T_{g^3h}(\tau),\\
\sum_{n\in \mathbb{Z}}\widetilde{\operatorname{Tr}}_5\big(h|\hat{H}^1(g,V_n\otimes \mathbb{Z}_5)\big)q^{n-1}= 0.
\end{gather*}

 By a similar calculation, for $g\in 21A$ and any 21-regular element $h\in C_{\mathbb{M}}(g)$, we have the following relation from Theorem \ref{15A21A}:
\begin{gather*}
\sum_{n\in \mathbb{Z}}\widetilde{\operatorname{Tr}}_3\big(h|\hat{H}^0(g,V_n\otimes \mathbb{Z}_3)\big)q^{n-1}= \frac{6}{7}T_{gh}(\tau)+\frac{1}{7}T_{g^7h}(\tau),\\
\sum_{n\in \mathbb{Z}}\widetilde{\operatorname{Tr}}_3\big(h|\hat{H}^1(g,V_n\otimes \mathbb{Z}_3)\big)q^{n-1}=0,\\
\sum_{n\in \mathbb{Z}}\widetilde{\operatorname{Tr}}_7\big(h|\hat{H}^0(g,V_n\otimes \mathbb{Z}_7)\big)q^{n-1}= \frac{2}{3}T_{gh}(\tau)+\frac{1}{3}T_{g^3h}(\tau),\\
\sum_{n\in \mathbb{Z}}\widetilde{\operatorname{Tr}}_7\big(h|\hat{H}^1(g,V_n\otimes \mathbb{Z}_7)\big)q^{n-1}= 0.
\end{gather*}

We propose the following conjecture about the graded $p$-Brauer character for Tate co\-ho\-mo\-logy.

\begin{Conjecture}
 For any element $g \in \mathbb{M}$ of order $N$ with prime factor $p$ and $N$-regular element $h \in C_{\mathbb{M}}(g)$, $\sum_{n\in \mathbb{Z}}\widetilde{\operatorname{Tr}}_p\big(h|\hat{H}^i(g,V_n\otimes \mathbb{Z}_p)\big)q^{n-1}$, $i=0,1$, is a~linear combination of Hauptmoduls.
\end{Conjecture}

If we have that the series
\[
\sum_{n\in \mathbb{Z}}\widetilde{\operatorname{Tr}}_p\big(h|\hat{H}^0(g,V_n\otimes \mathbb{Z}_p)\big)q^{n-1}+\sum_{n\in \mathbb{Z}}\widetilde{\operatorname{Tr}}_p\big(h|\hat{H}^1(g,V_n\otimes \mathbb{Z}_p)\big)q^{n-1}
\] is a linear combination of Hauptmoduls, then we describe clearly
\[
\sum_{n\in \mathbb{Z}}\widetilde{\operatorname{Tr}}_p\big(h|\hat{H}^i(g,V_n\otimes \mathbb{Z}_p)\big)q^{n-1}, \qquad i=0,1,
\]
by a linear combination of Hauptmoduls.

If conditions 2, 4 and 7 in Borcherds's conjecture at the end of~\cite{B98} are true, then $\hat{H}^1(g,V)$ vanishes for all Fricke elements~$g$.
From Corollary~\ref{newHaup}, we can give a counterexample to this conjecture of Borcherds about vanishing of Tate cohomology $\hat{H}^1(g,V)$.

\begin{Theorem}
Let $V$ be an integral form of the monster vertex algebra. Then there exists a~Fricke element $g$ s.t.\ $\hat{H}^1(g,V)\neq 0$.
\end{Theorem}

\begin{proof}
Let $g$ be an element of the monster in class $8A$. Then $g$ is a Fricke element. By Lemma~\ref{Zpkai} and Corollary~\ref{newHaup}, we have
\begin{align*}
\sum_{n \in \mathbb{Z}}\operatorname{length}_2\big(\hat{H}^*(g,V_n)\big)q^{n-1} &=a_{1,2}\varphi(8)T_g(\tau)+a_{2,2}\varphi(4)T_{g^2}(\tau)+a_{4,2}\varphi(2)T_{g^4}(\tau) \\
 &= 2T_{8A}(\tau)+\frac{3}{4}T_{4C}(\tau)+\frac{1}{4}T_{2B}(\tau),
\end{align*}
where $\operatorname{length}_2(\cdot)$ means length of a composition series of the $\mathbb{Z}_2$-module.
Comparing coefficients of $q^2$, we have
\[
\operatorname{length}_2\big(\hat{H}^0(g,V_3)\big)-\operatorname{length}_2\big(\hat{H}^1(g,V_3)\big)=-256.
\]
This means $\hat{H}^1(g,V_3)\neq 0$. Hence, $\hat{H}^1(g,V)\neq 0$.
\end{proof}

 By Corollary \ref{newHaup}, if $g$ is a Fricke element such that $g^j$ is a~Fricke element for any $j> 0$, then we have that all coefficients of $\sum_{n\in \mathbb{Z}}\operatorname{length}_{p_i}\big(\hat{H}^*(g,V_n\otimes \mathbb{Z}_{p_i})\big)q^{n-1}$ are positive, where $\operatorname{length}_{p_i}(\cdot)$ means length of a composition series of the $\mathbb{Z}_{p_i}$-module. Hence, we propose a weaker version of Borcherds's conjecture about vanishing of Tate cohomology $\hat{H}^1(g,V)$.

\begin{Conjecture}
Let $V$ be an self-dual integral form of the monster vertex algebra. If $g$ is a~Fricke element such that $g^j$ is a~Fricke element for any $j> 0$, then $\hat{H}^1(g,V)=0$.
\end{Conjecture}

 Suppose that $g$ has order $N$ and that the prime factorization of $N$ is $\prod p_i^{n_i}$. If we have that
\[
\operatorname{length}_{p_i}\big(\hat{H}^0(g,V_n\otimes\mathbb{Z}_{p_i})\big)-\operatorname{length}_{p_i}\big(\hat{H}^1(g,V_n\otimes\mathbb{Z}_{p_i})\big)
\]
 is equal to
\[
\operatorname{length}_{p_i}\big(\hat{H}^0(g,V_n\otimes\mathbb{Z}_{p_i})\big)+\operatorname{length}_{p_i}\big(\hat{H}^1(g,V_n\otimes\mathbb{Z}_{p_i})\big)
\]
for any $p_i$ and $n\in \mathbb{Z}$, then it will follow that $\hat{H}^1(g,V)=0$.

\subsection*{Acknowledgements}

I would like to thank Scott Carnahan for many helpful comments and advice. I would also like to thank the anonymous referees for many helpful comments.

\pdfbookmark[1]{References}{ref}
\LastPageEnding


\begin{thebibliography}{99}
\footnotesize\itemsep=0pt

\bibitem{B92}
Borcherds R.E., Monstrous moonshine and monstrous {L}ie superalgebras,
 \href{https://doi.org/10.1007/BF01232032}{\textit{Invent. Math.}} \textbf{109} (1992), 405--444.

\bibitem{B98}
Borcherds R.E., Modular moonshine.~{III}, \href{https://doi.org/10.1215/S0012-7094-98-09305-X}{\textit{Duke Math.~J.}} \textbf{93}
 (1998), 129--154, \href{https://arxiv.org/abs/math.AG/9801101}{arXiv:math.AG/9801101}.

\bibitem{BR}
Borcherds R.E., Ryba A.J.E., Modular {M}oonshine.~{II}, \href{https://doi.org/10.1215/S0012-7094-96-08315-5}{\textit{Duke Math.~J.}}
 \textbf{83} (1996), 435--459.

\bibitem{C}
Carnahan S., A self-dual integral form of the {M}oonshine module,
 \href{https://doi.org/10.3842/SIGMA.2019.030}{\textit{SIGMA}} \textbf{15} (2019), 030, 36~pages, \href{https://arxiv.org/abs/1710.00737}{arXiv:1710.00737}.

\bibitem{CF}
Cassels J.W.S., Fr\"ohlich A. (Editors), Algebraic number theory, Academic
 Press, London, Thompson Book Co., Inc., Washington, D.C., 1967.

\bibitem{CN}
Conway J.H., Norton S.P., Monstrous moonshine, \href{https://doi.org/10.1112/blms/11.3.308}{\textit{Bull. London Math. Soc.}}
 \textbf{11} (1979), 308--339.

\bibitem{FLM}
Frenkel I., Lepowsky J., Meurman A., Vertex operator algebras and the
 {M}onster, \textit{Pure and Applied Mathematics}, Vol. 134, Academic Press,
 Inc., Boston, MA, 1988.

\end{thebibliography}
\end{document}